\documentclass[11pt]{amsart}
\usepackage[adobe-utopia]{mathdesign}
\usepackage[T1]{fontenc}
\usepackage{url}

\usepackage[vmargin=2.5cm,hmargin=2.5cm]{geometry}
\usepackage{color}
\usepackage{enumerate}

\newtheorem{Thm}{Theorem}
\newtheorem{Problem}{Problem}
\newtheorem{Con}{Conjecture}
\newtheorem{Lem}{Lemma}
\newtheorem{Remark}{Remark}

\newtheorem{Def}{Definition}
\newtheorem{cor}{Corollary}
\newtheorem{Prop}{Proposition}
\newtheorem{algorithm}{Algorithm}

\theoremstyle{remark}

\newtheorem{Example}{Example}

\DeclareMathOperator{\Ap}{Ap}

\title{Cyclotomic numerical semigroups}
\author{Emil-Alexandru Ciolan}
\address{Rheinische Friedrich-Wilhelms-Universit\"at Bonn, Regina-Pacis-Weg 3, D-53113 Bonn, Germany}
\email{calexandru92@yahoo.com}

\author{Pedro A. Garc\'ia-S\'anchez}
\address{Departamento de \'Algebra, Universidad de Granada, E-18071 Granada, Espa\~na}
\email{pedro@ugr.es}
\thanks{The second author is supported by the projects MTM2010-15595, MTM2014-55367-P, FQM-343,  FQM-5849, and FEDER funds}

\author{Pieter Moree}
\address{Max-Planck-Institut f\"ur Mathematik, Vivatsgasse 7, D-53111 Bonn, Germany}
\email{moree@mpim-bonn.mpg.de}
\begin{document}
\date{}

\begin{abstract}
\noindent Given a numerical semigroup $S$, we let $\mathrm P_S(x)=(1-x)\sum_{s\in S}x^s$ be its {\it semigroup polynomial}. We
study {\it cyclotomic numerical semigroups}; these are numerical semigroups 
$S$ such that $\mathrm P_S(x)$ has all its roots in the unit disc. We conjecture
that $S$ is a cyclotomic numerical semigroup if and only if $S$ is a complete intersection 
numerical semigroup and present some evidence for it.\\
\indent Aside from the notion of cyclotomic numerical semigroups we
introduce the notion of \textit{cyclotomic exponents} and \textit{polynomially
related numerical semigroups}. We derive some properties and give
some applications of these new concepts.
\end{abstract}
\maketitle
\section{Introduction}
A \emph{numerical semigroup} $S$ is a submonoid of $\mathbb N$ (the set of nonnegative integers) under addition, with finite complement in $\mathbb N$. The nonnegative integers not in $S$ are its \emph{gaps}, and the largest integer not in $S$ is its \emph{Frobenius number}, $\mathrm F(S)$. The number of gaps of $ S, $ also known as the \textit{genus} of $ S ,$ is denoted by $ \mathrm g(S). $ A numerical semigroup admits a unique minimal generating system; its cardinality is called its \emph{embedding dimension} $\mathrm e(S)$, and its elements \emph{minimal generators}. The smallest positive integer in $S$ is called the \emph{multiplicity} of $S$, and it is denoted by $\mathrm m(S)$ (see for instance \cite{roos} for an introduction to numerical semigroups). 

For $A\subseteq\mathbb N$, we use $\langle A\rangle$ to denote the set of integers of the form $\sum_{a\in A}\lambda_a a$ where $a\in A$, $\lambda_a\in\mathbb N$ and all but finitely many $ \lambda_a $ are equal to zero. The set $\langle A\rangle$ is a numerical semigroup if and only if the greatest common divisor of the elements of $A$ equals 1.

To a numerical semigroup $S$ we can associate
$\mathrm H_S(x):=\sum_{s\in S}x^s,$ its {\it Hilbert series} 
(sometimes called the generating function associated to $S$), and $\mathrm P_S(x)=(1-x)\sum_{s\in S}x^s,$ its {\it semigroup polynomial}.
Since all elements larger than $\mathrm F(S)$ are in $S$, $\mathrm H_S(x)$ is not a polynomial, but
$\mathrm P_S(x)$ is. 
On noting that $\mathrm H_S(x)=(1-x)^{-1}-\sum_{s\not\in S}x^s$, we see that
\begin{equation}
\label{psx}
\mathrm P_S(x)=1+(x-1)\sum_{s\not \in S}x^s,
\end{equation}
where $s\not\in S$ denotes the sum over the numbers in
${\mathbb N}\backslash S$.
Observe that $\mathrm P_S(x)$ is a monic polynomial of degree $\mathrm F(S)+1$.  

Recall (see, for instance, Damianou \cite{Dam}) that a \emph{Kronecker polynomial} is a monic polynomial with integer coefficients having all its roots in the unit disc.
We define a numerical semigroup to be \emph{cyclotomic} if its semigroup polynomial is a Kronecker polynomial.
The following result of Kronecker 
and the fact that $\mathrm P_S(1)\ne 0$ allow us to give an alternative more explicit
definition (readers not so familiar with cyclotomic polynomials are 
referred to Section \ref{cpo}).
\begin{Lem} [Kronecker, 1857, cf. \cite{Dam}]
\label{leo}
If $f$ is a Kronecker polynomial with $f(0)\ne 0$, then all roots
of $f$ are actually on the unit circle and $f$ factorizes over
the rationals as a product of cyclotomic polynomials.
\end{Lem}
\begin{Def}
\label{cy}
A numerical semigroup is \emph{cyclotomic} if its semigroup polynomial 
factorizes over the rational numbers as a product of cyclotomic polynomials, 
that is, if we can write 
\begin{equation}
\label{basaal}
\mathrm P_S(x)=\prod_{d\in {\mathcal D}}\Phi_d(x)^{e_d},
\end{equation} 
with  ${\mathcal D}$ a finite set of positive integers and every $e_d$ a positive integer.
\end{Def}
\begin{Remark}
Since  cyclotomic polynomials are irreducible, the exponents
$e_d$ are unique.
\end{Remark}
\begin{Remark}
\label{ien}
On using that $\Phi_n$ is selfreciprocal for $n>1$ and that
$\Phi_1$ does not divide $\mathrm P_S$, we infer that 
if $S$ is cyclotomic, then $\mathrm P_S$ is selfreciprocal.
\end{Remark}
We can now formulate the main problem we like to address:
\begin{Problem}
\label{mainproblemo}
Find an intrinsic characterization of the numerical semigroups $S$ for which $S$ is cyclotomic, that is, one which does not involve $\mathrm P_S$ or its roots in any way.
\end{Problem}
Our conjectural solution of this problem involves two classes of numerical
semigroups: the symmetric ones and the complete intersection ones.\\
\indent Recall that a numerical semigroup $S$ is said to be \emph{symmetric} if $S\cup (\mathrm F(S)-S)=\mathbb Z$, thus 
symmetry is an example of an intrinsic characterization of $S$.
\begin{Thm}
\label{cycsym}
If $S$ is cyclotomic, then it must be symmetric.
\end{Thm}
\begin{proof} Using \eqref{psx} it is not difficult to conclude (see Moree \cite{Moree}) that
$S$ is symmetric if and only if $\mathrm P_S$ is selfreciprocal. By Remark \ref{ien}, $\mathrm P_S$ is selfreciprocal. 
\end{proof}
The converse is however not true, as illustrated at the end of Section \ref{conj+sym}.

\indent New let us recall the definition of \emph{complete intersection} numerical semigroups. If $S$ is minimally generated by $\{n_1,\ldots,n_e\}$, then the monoid morphism $\phi: \mathbb N^e\to S$, $\phi(a_1,\ldots,a_e)=\sum_{i=1}^e a_in_i$ is an epimorphism. Consequently $S$ is isomorphic, as a monoid, to $\mathbb N^e/\ker \phi$, where $\ker\phi=\{(a,b)\in\mathbb N^e\times \mathbb N^e: \phi(a)=\phi(b)\}$ is the kernel congruence of $\phi$. It turns out that $\ker\phi$ is finitely generated (as a congruence) and that the minimum number of generators is at least $\mathrm e(S)-1$ (see for instance \cite{roos}). Then $S$ is a \emph{complete intersection} numerical semigroup if $\ker \phi$ is minimally generated by $\mathrm e(S)-1$ pairs. It is well-known that every complete intersection numerical semigroup is symmetric (see for instance \cite[Corollary 9.17]{roos}).
\par The following observation is not deep, see Section \ref{glueglue}.
\begin{Lem}
\label{ciisc}
Every complete intersection numerical semigroup is cyclotomic.
\end{Lem}
The next lemma sums up the above\footnote{The referee suggested the following mnemonic as an easy way to remember the order of inclusions. Look at the initials of the words and sort them alphabetically: $ \text{CI}\subseteq\text{CY}\subseteq\text{SY}. $ }.
\begin{Lem}
\label{inclusie}
We have the following inclusions of numerical semigroups
$$\{\text{complete intersection}\}\subseteq \{\text{cyclotomic}\} \subseteq 
\{\text{symmetric}\}.$$
\end{Lem}
Let ${\mathcal I}_k$, ${\mathcal C}_k$ 
and ${\mathcal S}_k$ denote the set of numerical semigroups that have
Frobenius number $k$ and are complete intersection, cyclotomic and 
symmetric, respectively.
Using the GAP package \cite{numericalsgps} it is seen that
   ${\mathcal I}_k={\mathcal C}_k$ for $k\le 70$. If in the sequel
we state that a polynomial is Kronecker or not, this
was always established using this package 
(using the Graeffe method based on \cite{effective}).

We conjecture that the first two sets in Lemma 
\ref{inclusie} are actually equal (that is, ${\mathcal I}_k={\mathcal C}_k$ for every $k\ge 1$).
\begin{Con}\label{conjecture}
A numerical semigroup is cyclotomic if and only if it is a complete intersection numerical semigroup.
\end{Con}
The second set, however, is strictly contained in the third one (we have, e.g., 
${\mathcal C}_{9}\subsetneq {\mathcal S}_9$, see Section \ref{conj+sym}).
Here we make the following conjecture.
\begin{Con}
Let $e\ge 4$. There exists a symmetric numerical semigroup of
embedding dimension $e$ that is not cyclotomic.
\end{Con}
In the rest of the paper, the main focus is
on theoretical contributions towards solving Conjecture \ref{conjecture}
and some related problems.  

Throughout, the letters $p,q$ and $r$ are used to indicate primes. 

For a pedestrian introduction to both cyclotomic polynomials and 
numerical semigroups, the 
reader is referred to Moree \cite{Moree}.

\section{Tools}
\subsection{Cyclotomic polynomials}
\label{cpo}
In this section we discuss relevant (elementary) properties of cyclotomic polynomials.
A nice introduction to cyclotomic polynomials
is Thangadurai \cite{tanga}.\\
We let
$$\Phi_n(x)=\prod_{j=1,~(j,n)=1}^n\left(1-e^{\frac{2\pi ij}{n}}\right)$$
denote the $n$-th cyclotomic polynomial. It is well-known that it has integer
coefficients. Furthermore it is monic  
of degree $\varphi(n)$ (where $\varphi$ denotes Euler's totient function) and irreducible over $\mathbb Q$ (see, e.g., Weintraub \cite{wein}).
Over the rational numbers $x^m-1$ factorizes into irreducibles as
\begin{equation}
\label{uitelkaar}
x^m-1=\prod_{d\mid m}\Phi_d(x).
\end{equation}
This equation implies that $\Phi_d$ divides $x^n-1$ if and only if $d$ divides $n$.
By M\"obius inversion we infer from \eqref{uitelkaar} that
\begin{equation}
\label{prod1}
\Phi_n(x)=\prod_{d\mid n}(x^d-1)^{\mu(n/d)},
\end{equation}
where $\mu(n)$ denotes the M\"obius function. 
It follows, for example from \eqref{prod1}, that
if $p$ and $q$ are distinct primes, then 
\begin{equation}
\label{basic-2}
\Phi_{pq}(x)=\frac{(x^{pq}-1)(x-1)}{(x^p-1)(x^q-1)}.
\end{equation} 
Using \eqref{prod1} it is easily shown 
that
\begin{equation}
\label{flauuw}
\Phi_{pn}(x)=\Phi_n(x^p){\rm ~if~}p\mid n.
\end{equation}

On invoking the fact that $\sum_{d\mid n}\mu(d)=0$ for $n>1$, we infer from \eqref{prod1} that
\begin{equation}
\label{phi-signchange}
\Phi_n(x)=\prod_{d\mid n}(1-x^d)^{\mu(n/d)},
\end{equation}
and, on using the identity $\sum_{d\mid n}d\mu(n/d)=\varphi(n)$, we deduce that
\begin{equation}\label{phi-selfrecirpcal}
x^{\varphi(n)}\Phi_n\left(\frac{1}{x}\right)=\Phi_n(x).
\end{equation}
Hence $\Phi_n$ is selfreciprocal for $n>1$. Note that $\Phi_1(x)=x-1$ is not selfreciprocal.\\
\indent It is a well-known fact, see, e.g., Lang \cite[p. 74]{Lang}, that
\begin{equation}
\label{pn1}
\Phi_n(1)=\begin{cases}0 & \hbox{ if } n=1;\\
p & \hbox{ if } n=p^m;\\
1 & \hbox{otherwise}.\end{cases}
\end{equation}
Note that $\Phi_1(-1)=-2$ and $\Phi_2(-1)=0$. For $n>2$ we have
\begin{equation}
\label{pn2}
\Phi_n(-1)=\begin{cases}
p & \hbox{ if } n=2p^m;\\
1 & \hbox{otherwise},\end{cases}
\end{equation}
which follows from (\ref{flauuw}) and
the observation that $\Phi_{2m}(x)=\Phi_m(-x)$ if $m>1$ is odd (for a
different proof, see \cite{Dresden}).

\subsection{Semigroup polynomials.} In this section we establish some basic, yet useful facts relating a numerical semigroup to its polynomial.
\begin{Lem}\label{coeff-PS}
Let $S$ be a numerical semigroup and assume that $\mathrm P_S(x)=a_0+a_1x+\cdots+ a_kx^k$. Then, for $s\in \{0,\ldots, k\}$,
\[
a_s=\begin{cases}
1 & \hbox{if } s\in S\hbox{ and }s-1\not\in S;\\
-1 & \hbox{if } s\not\in S\hbox{ and }s-1\in S;\\
0 &\hbox{otherwise}.
\end{cases}
\]
\end{Lem}
\begin{proof}
The proof easily follows from the fact that $\mathrm P_S(x)=(1-x)\mathrm H_S(x)$ and that a coefficient of $x\mathrm H_S(x)$ is 1 if and only if its degree minus one belongs to $S$.
\end{proof}

\begin{cor}
\label{corrie}
The nonzero coefficients of $\mathrm P_S(x)$ alternate between $1$ and $-1$.
\end{cor}
\begin{Lem}
\label{gaps}
Let $S\ne \mathbb N$ be a numerical semigroup. Then 
\begin{enumerate}[a)]
\item $ \mathrm{P}_S(1)=1; $
\item $ \mathrm P_S'(1)=\mathrm g(S). $
\end{enumerate}
\end{Lem}
\begin{proof}~\\
a) Immediate from \eqref{psx}.\\
b) The condition $ S\ne \mathbb N $ is equivalent to $ 1\notin S, $ ensuring $ \mathrm m(S)\ge 2. $ On using Lemma \ref{coeff-PS} together with the fact that the first nonzero element of $ S $ is $ \mathrm m(S) $ and the last gap of $ S $ is $ \mathrm F(S), $  we note that there exist $ 2\le k_1<\cdots<k_{2n+1} $ such that 
\begin{equation}
\label{maybe}
\mathrm P_S(x)=1-x+x^{k_1}-x^{k_2}+\cdots-x^{k_{2n}}+x^{k_{2n+1}}.
\end{equation}
In fact, $k_1=\mathrm m(S) $ and $ k_{2n+1}=\mathrm F(S)+1. $ Lemma \ref{coeff-PS} tells us that \begin{equation}
\label{gapss}
\mathbb N\backslash S= [1,k_1-1]\cup[k_2,k_3-1]\cup\ldots\cup[k_{2n},k_{2n+1}-1] 
\end{equation} Differentiating \eqref{maybe} yields 
\[\mathrm P_S'(x)=(-1+k_1x^{k_1-1})+\cdots+(-k_{2n}x^{k_{2n}-1}+k_{2n+1}x^{k_{2n+1}-1})\] and evaluating at 1 gives 
\begin{equation}\label{gapsss}
\mathrm P_S'(1)=(k_1-1)+(k_3-k_2)+\cdots+(k_{2n+1}-k_{2n}).\end{equation}
The conclusion now follows on comparing \eqref{gapss} and \eqref{gapsss}.
\end{proof}
Note that
\begin{equation}
\label{atmineen}
\mathrm P_S(-1)=1+2g(S)-4\sum_{s\not\in S\atop 2\mid s}1
\end{equation}
is an odd integer and hence nonzero.
\begin{Lem}\label{pieter-cond}
If $S$ is a numerical semigroup satisfying
$$\sum_{s\not\in S,~2\nmid s}1<\sum_{s\not\in S,~2|s}1,$$
then $S$ is not cyclotomic. 
\end{Lem}
\begin{proof}
On using (\ref{atmineen}) the latter inequality is 
seen to be equivalent with $\mathrm P_S(-1)<0$. Now assume
that $S$ is cyclotomic.
Then
by (\ref{pn2}) always $\Phi_n(-1)\ge 0$ and hence $\mathrm P_S(-1)\ge 0$.
This contradiction finishes the proof.
\end{proof}

\begin{Example}
Take $S=\langle 3,5,7\rangle$. It has one odd gap and two even gaps. By Lemma \ref{pieter-cond}, $S$ is not cyclotomic. Observe that $S$ is not even symmetric. 

The converse of Lemma \ref{pieter-cond} is not true. The numerical semigroup $S=\langle 5, 6, 7, 8\rangle$ is not cyclotomic; there are three odd gaps in $S$ and 
two even gaps.

We took all numerical semigroups $S$ that are symmetric and not complete intersection with Frobenius number $\le k$ and determined how often on average 
Lemma \ref{pieter-cond} applies. Our computations 
(with $k\le 69$) indicate that likely an average
exists and is in $[0.8,0.85]$. 
\end{Example}

\subsection{Ap\'ery sets and semigroup polynomials}
 The {\it Ap\'ery set} of $S$ with respect to a nonzero $m\in S$ is defined as
$$\Ap(S;m)=\{s\in S \colon  s-m\not\in S\}.$$
Note that
\begin{equation}
\label{standar-basis}
S=\Ap(S;m)+m\mathbb N
\end{equation}
and that $\Ap(S;m)$ consists of a complete set of residues modulo $m$. Thus we
have
\begin{equation}
\label{oudedoos}
\mathrm H_S(x)=\sum_{w\in \Ap(S;m)}x^w\sum_{i=0}^{\infty}x^{mi}=\frac{1}{1-x^m}\sum_{w\in \Ap(S;m)}x^w,
\end{equation}
cf. \cite[(4)]{RR}.

Ap\'ery sets can also be defined in a natural way for integers $m$ not in the semigroup (see for instance \cite{fr} or \cite{gs-l}), but in this case $\#\Ap(S;m)\neq m$.

\begin{Prop}\label{card-apery}
Let $S$ be a numerical semigroup and $m$ be a positive integer. Then $\#\Ap(S;m)=m$ if and only if $m\in S$.
\end{Prop}
\begin{proof} 
For $i\in \{0,\ldots, m-1\}$ set $w_i=\min\{ s\in S \colon  s\equiv i \pmod m \}$. By definition, $w_0=0$ and $\{w_0,\ldots, w_{m-1}\}\subseteq \Ap(S;m)$. Hence $\#\Ap(S;m)\ge m$, and equality holds if and only if $\{w_0,w_1,\ldots, w_{m-1}\}=\Ap(S;m)$.

If $m\in S$, \cite[Lemma 2.4]{roos} asserts that $\{w_0,w_1,\ldots, w_{m-1}\}=\Ap(S;m)$.

Now assume that $\{w_0,w_1,\ldots, w_{m-1}\}=\Ap(S;m)$. Then, for every $i\in \{0,\ldots,m-1\}$ and every $k\in \mathbb N$, $w_i+km\in S$. In particular $w_0+m=m\in S$.
\end{proof}

\begin{Example}
Let $S$ be a numerical semigroup minimally generated by $\{a,b\}$. Assume that $u,v$ are integers with $0\le u<b$ and $1=ua+vb$. By Lemma \ref{coeff-PS}, the number of ones in $\mathrm P_S$ equals $\#\Ap(S;1)$ and, in view of \cite[Theorem 14]{fr}, we have that $\#\Ap(S;1)=u(a+v)$ (compare with \cite[Corollary 1]{Moree}). Given $0<\gamma<1/2$, let $\mathrm C_{\gamma}(x)$ 
denote the number of numerical semigroups $S=\langle p,q\rangle$ with $p,q$ primes and $m=pq\le x$ such that
$\#\Ap(\langle p,q \rangle;1)\le m^{1/2+\gamma}$. Bzd{\k{e}}ga \cite{Bartek} was the first to 
obtain sharp upper and lower bounds for this quantity.
Fouvry \cite{Fouvry}, using deep methods from analytic number theory, even obtained an asymptotic for
$\mathrm C_{\gamma}(x)$ in the range $\gamma\in \left(\frac{12}{25},\frac{1}{2}\right)$.
\end{Example}

\begin{Example}\label{sym1} 
Let $m$ and $q$ be positive integers such that $m\ge 2q+3$ and let 
$$S=\langle m,m+1,qm+2q+2,\ldots,qm+(m-1)\rangle.$$ Then by \cite[Lemma 4.22]{roos} $S$ is symmetric with multiplicity
$m$ and embedding dimension $m-2q$. It is easy to deduce that $\Ap(S;m)  =  \{0,m+1,2m+2,\ldots,qm+q,qm+2q+2,\ldots,qm+(m-1),(q+1)(m+1), \ldots,(2q+1)(m+1)\}$.
On invoking \eqref{oudedoos} and carrying out the computations, we obtain an explicit formula for the semigroup polynomial:
$$\mathrm P_S(x)=\sum_{k=0}^q x^{km}+x^{q(m+2)+2}\sum_{k=0}^{q+1}x^{km}-x\sum_{k=0}^{2q+1}x^{k(m+1)}.$$


\end{Example}

%
%

\section{Conjecture \ref{conjecture} holds for embedding dimension $\le 3$}\label{conj+sym}
Using the fact that every symmetric numerical semigroup $ S $ with embedding dimension $\mathrm e(S)\le 3$ is a complete intersection 
(\cite{Herzog}), it is easy to see that the following result
holds.
\begin{Lem}
\label{zeven}
For all numerical semigroups $ S $ with $\mathrm e(S)\le 3$, we have
$$\text{complete intersection}\Leftrightarrow\text{cyclotomic}\Leftrightarrow \text{symmetric}.$$
\end{Lem}

\begin{Example}
Let $S=\langle 4,6,9\rangle$. We find
that $S$ is symmetric and hence it must be cyclotomic. Indeed, we
have $\mathrm P_S(x)=x^{12}-x^{11}+x^8-x^7+x^6-x^5+x^4-x+1= \Phi_6(x)\Phi_{12}(x)\Phi_{18}(x)$.
\end{Example}

\begin{cor}
Conjecture \ref{conjecture} holds true for all numerical semigroups $ S $ such that $ \mathrm{e}(S)\le 3. $ 
\end{cor}
The analogous version of Lemma \ref{zeven} is not true if $\mathrm e(S)=4$ as shown, for instance, by the numerical semigroup $ S=\langle 6,7,10,11\rangle $ that is obtained by setting $ m=6 $ and $ q=1 $ in Example \ref{sym1}. The semigroup polynomial then equals $\mathrm P_S(x)=x^{16}-x^{15}+x^{10}-x^8+x^6-x+1$, which is not Kronecker. Further, we suspect that the numerical semigroups described in Example \ref{sym1} are not cyclotomic for embedding dimension $ \ge 4. $ We did an exhaustive search in this family of numerical semigroups up to multiplicity 30, and indeed, only those with embedding dimension three were cyclotomic.   
\par It turns out that the smallest Frobenius number that can occur for
a symmetric numerical semigroup that is not cyclotomic is $9$. There is
only one such semigroup, namely 
$S=\langle  5, 6, 7, 8 \rangle$, where we have $\mathrm P_S(x)=x^{10}-x^9+x^5-x+1$. For Frobenius number 11, we have two symmetric numerical semigroups that are not cyclotomic:  $\langle 5, 7, 8, 9 \rangle$ and $\langle 6, 7, 8, 9, 10\rangle$. (Recall that a symmetric numerical semigroup has an odd Frobenius number, see, for instance, \cite{roos}.)  
\begin{Problem}
Prove that the numerical semigroups $ S $ given in Example 
\ref{sym1} for which $\mathrm e(S)\ge 4$ are not cyclotomic, or find a counterexample.
\end{Problem}

\section{On the factorization of $\mathrm P_S(x)$ into irreducibles}
In this section we consider a cyclotomic numerical semigroup
$S$ and try to infer some restrictions on the possible factorizations
of $\mathrm P_S(x)$ into cyclotomic polynomials. Lemma \ref{lem-basaal} 
is obtained on substituting $x=1$, and Lemma \ref{newbe} on substituting
$x=-1$ in (\ref{basaal}).
\begin{Lem}\label{lem-basaal}
\label{begin}
Let  $S$ be cyclotomic and ${\mathcal D}$ be as in Definition \ref{cy}.
If $d\in {\mathcal D}$, then $d>1$ and $d$ is not a prime power.
\end{Lem}
\begin{proof} 
By \eqref{psx} and \eqref{basaal} we have $1=\mathrm P_S(1)=\prod_{d\in {\mathcal D}}\Phi_d(1)^{e_d},$
and hence $e_1=0$. The proof is completed
on using that $\Phi_d(1)\not\in \{-1,1\}$ for those $d$ that 
are prime powers (see (\ref{pn1})). 
\end{proof}
Recall that
$\mathrm P_S(-1)=1+2\rm g(S)-4\sum_{s\not\in S,~2\mid s}1$. This implies in particular that $ \mathrm P_S(-1) $ is odd.
\begin{Lem}
\label{newbe}
Let $S$ be a cyclotomic numerical semigroup and $p>2$ a prime. Then
$$p\mid \mathrm  P_S(-1)\Leftrightarrow 
\Phi_{2p^k}(x)\mid\mathrm P_S(x)$$ 
for some $k\ge 1.$
\end{Lem}
\begin{proof}
"$\Leftarrow$". The assumption 
$\Phi_{2p^k}(x)\mid \mathrm P_S(x)$ implies that $\Phi_{2p^k}(-1)\mid\mathrm P_S(-1)$.  
Now invoke (\ref{pn2}) and (\ref{atmineen}).\\
"$\Rightarrow$". We must have $p\mid\Phi_n(-1)$ for some $n$ and
$\Phi_n(x)\mid\mathrm P_S(x)$. 
By Lemma \ref{lem-basaal} we must have $n>2$ (in fact $n\ge 6$) and
$n$ is not a power of two.
By (\ref{pn2}) it now follows that $n=2p^k$ for some $k\ge 1.$ 
\end{proof}

\begin{Example}
Let $S=\langle 6,9,11\rangle$. Then 
\[\mathrm P_S(x)= x^{26}-x^{25}+x^{20}-x^{19}+x^{17}-x^{16}+x^{15}-x^{13}+x^{11}-x^{10}+x^9-x^7+x^6-x+1 =\Phi_{18}(x)\Phi_{33}(x).
\] 
Observe that $18=2\cdot 3^2$ and 
$\mathrm P_S(-1)=3$.
\end{Example}
\begin{Remark}
If $p$ divides $\mathrm P_S(-1)$ exactly, then there is a unique positive integer
$k$ such that $\Phi_{2p^k}(x)$ divides $\mathrm P_S(x)$ exactly.
\end{Remark}

\section{Gluings of numerical semigroups}
\label{glueglue}

In this section we will use gluings to infer that every complete intersection
numerical semigroup is cyclotomic and hence Lemma \ref{ciisc}.

Let $T$, $T_1$ and $T_2$ be submonoids of $\mathbb N$. We say that $T$ is the \textit{gluing} of $T_1$ and $T_2$ if 
\begin{enumerate}
\item $T=T_1+T_2$;
\item $\mathrm{lcm}(d_1,d_2)\in T_1\cap T_2$, with $d_i=\gcd(T_i)$ for $i\in \{1,2\}$,
\end{enumerate}
and we will write $T=T_1+_d T_2$, with $d=\mathrm{lcm}(d_1,d_2)$. 

Every nontrival submonoid $T$ of $\mathbb N$ is isomorphic as a monoid to $T/\gcd(T)$, which is a numerical semigroup. Hence, in the above definition if $T=S$ is a numerical semigroup, and $S=T_1+_d T_2$, then $T_i=d_iS_i$, with $S_i=T_i/d_i$, and $\gcd(d_1,d_2)=\gcd(S)=1$. Hence $\mathrm{lcm}(d_1,d_2)=d_1d_2$, which leads to $d_i\in S_j$ for $\{i,j\}=\{1,2\}$.


For $a_1,a_2$ integers greater than 2 with $\gcd(a_1,a_2)=1$, it is shown in \cite{gluings-fv-hs} that 
\begin{equation}\label{hil-gluing}
\mathrm H_{a_1S_1+_{a_1a_2} a_2S_2}(x)=(1-x^{a_1a_2})\mathrm H_{S_1}(x^{a_1})\mathrm H_{S_2}(x^{a_2}).
\end{equation} 
For the particular case $S=\langle a_1,a_2\rangle=a_1\mathbb N+_{a_1a_2} a_2\mathbb N$, we obtain (see also \cite{Moree})
\begin{equation}\label{hil-embd-2}
\mathrm H_{\langle a_1,a_2\rangle}(x)=\frac{1-x^{a_1a_2}}{(1-x^{a_1})(1-x^{a_2})},
\end{equation}
and by using \eqref{uitelkaar}, we get
\begin{equation} \label{Pa1a2}
\mathrm P_{\langle a_1,a_2\rangle}(x)=\frac{(1-x)(1-x^{a_1a_2})}{(1-x^{a_1})(1-x^{a_2})}= \prod_{d\mid a_1a_2,~d\nmid a_1,~d\nmid a_2}\Phi_d(x). 
\end{equation}
Hence
\begin{equation}\label{pol-gluing}
\mathrm P_{a_1S_1+_{a_1a_2} a_2S_2}(x)=\frac{(1-x)(1-x^{a_1a_2})}{(1-x^{a_1})(1-x^{a_2})}\mathrm P_{S_1}(x^{a_1})\mathrm P_{S_2}(x^{a_2})=\mathrm P_{\langle a_1,a_2\rangle}(x) \mathrm P_{S_1}(x^{a_1})\mathrm P_{S_2}(x^{a_2}).
\end{equation}

Delorme in \cite{Delorme} proved (but with a different notation) that a numerical semigroup $S$ is a complete intersection if and only if $S$ is either $\mathbb N$ or the gluing of two complete intersection numerical semigroups. If we proceed recursively and $A=\{a_1,\ldots, a_t\}$ is a minimal generating system of $S$, we will find positive integers $g_1,\ldots,g_{t-1}$ such that 
\[S=a_1\mathbb N+_{g_1} a_2\mathbb N+\cdots +_{g_{t-1}} a_t\mathbb N.\]
By using \cite[Theorem 20]{gluings-fv-hs}, we obtain
\begin{equation}\label{pol-ci}
\mathrm P_S(x) =(1-x) {\prod_{i=1}^{t-1}(1- x^{g_i})}\prod_{i=1}^t(1- x^{a_i})^{-1},
\end{equation}
and we deduce that
every complete intersection numerical semigroup is cyclotomic, and
hence we have proved Lemma \ref{ciisc}, that is, one of the directions of Conjecture \ref{conjecture}. 
\par For $S=\langle n_1,\ldots,n_e\rangle$, according to \cite[(1)]{s-w}, the only nonzero terms of $\mathcal K(x)=\mathrm H_S(x)\prod_{i=1}^e (1-x^{n_i})$ are those of degrees $n\in S$ such that the \textit{Euler characteristic} of the \textit{shaded set} of $n$, $\Delta_n = \{ L\subset \{n_1,\ldots, n_e\} \colon n-\sum_{s\in L} s\in S\}$, is not zero, that is, $\chi_S(n):=\sum_{L\in \Delta_n}(-1)^{\# L}\neq 0$. We have been trying to determine whether $\mathcal K(x)$ factors as $\prod_{b\in \mathrm{Betti}(S)}(1-x^b)^{m_b}$, where $\mathrm{Betti}(S)$ is the set of the \textit{Betti numbers} of $ S $, i.e., the elements $ n\in S $ for which the underlying graph of $\Delta_n$ is not connected (the graph whose vertices are the elements $n_i\in\{n_1,\ldots,n_e\}$ such that $n-n_i\in S$, and $n_in_j$ is an edge whenever $i,j\in \{1,\ldots, e\}$, $i\neq j$ and $n-(n_i+n_j)\in S$; see \cite[\S 7.3]{roos}) and $m_b\in\mathbb{N}$. This is what actually happens in \eqref{pol-ci}. We will detail our efforts done in this regard in Section \ref{Beatrice}, after introducing some further tools.

\subsection{Free semigroups}\label{free-semigroups}

Let $S$ be a numerical semigroup generated by $\{n_1,\ldots,n_t\}$. We say that $S$ is \emph{free} if either $S=\mathbb N$ or it is the gluing of the free semigroup $\langle n_1,\ldots, n_{t-1}\rangle$ and $\langle n_t\rangle$ (see \cite{b-c}). The way we enumerate the generators is relevant. For instance $S$ is free for the arrangement $\{n_1=4,n_2=6,n_3=9\}$, but not for $\{n_1=4, n_2=9, n_3=6\}$.

\begin{Example}\label{ex-f-sem}
Let $S$ be an embedding dimension three symmetric numerical semigroup. Then $S$ is free and it has a system of generators of the form $\langle am_1,am_2,b m_1 + cm_2\rangle$, with $a,b,c\in \mathbb N$ such that $a\ge 2$, $b+c\ge 2$ and $\gcd(a,bm_1+cm_2)=1$ (\cite[Theorem 10.6]{roos}). It follows that $S$ can be expressed as $S=a\langle m_1,m_2\rangle +_{a(bm_1+cm_2)} (bm_1+cm_2)\mathbb N$. From \eqref{pol-ci} we get 
\[\mathrm P_S(x)= \frac{(1-x)(1-x^{a(bm_1+cm_2)})(1-x^{am_1m_2})}{(1-x^{am_1})(1-x^{am_2})(1-x^{bm_1+cm_2})}.\]
If $S=\langle n_1,n_2,n_3\rangle$ is nonsymmetric with embedding dimension three, then it can be deduced from \cite{s-w} and \cite{Betti} (see also \cite[Theorem 4]{RR}) that 
\[\mathrm P_S(x) =\frac{(1-x)(1-x^{c_1n_1}-x^{c_2n_2}-x^{c_3n_3}+x^{f_1+n_1+n_2+n_3}+x^{f_2+n_1+n_2+n_3})}{(1-x^{n_1})(1-x^{n_2})(1-x^{n_3})},\]
where 
\begin{itemize}
\item $c_i=\min\{m \in \mathbb N\setminus\{0\} \colon mn_i\in \langle n_j,n_k\rangle \}$ for all $\{i,j,k\}=\{1,2,3\}$,
\item $f_1=\mathrm F(S)$ and $f_2\neq f_1$ is such that $f_2+S\setminus \{0\}\subset S$ ($f_1$ and $f_2$ are the \textit{pseudo-Frobenius numbers} of $S$; their expression can be found for instance in \cite[Corollary 11]{Betti}).
\end{itemize}
Formulas for symmetric and pseudo-symmetric embedding dimension four can be derived from \cite[Section 4]{Betti}, and the number of nonzero coefficients of $\mathrm H_S(x)\prod_{i=1}^4(1-x^n_i)$ is 12 and 14, respectively. (Recall that $S$ is \textit{pseudo-symmetric} if $\mathrm F(S)$ is even and for every $x\in\mathbb Z\setminus S$, either $x=\mathrm F(S)/2$ or $\mathrm F(S)-x\in S$.) {}From \cite{s-w} it follows that the number of nonzero coefficients 
of $\mathrm H_S(x)\prod_{i=1}^4(1-x^n_i)$ is not bounded when $S$ ranges
over all numerical semigroups of embedding dimension $4$.
\end{Example}
Special families of free numerical semigroups are the \textit{telescopic} ones (free with respect to the arrangement $n_1<n_2<\cdots <n_t$, \cite{telescopic}), 
numerical semigroups associated to irreducible plane curve singularities (\cite{zar}) and \textit{binomial semigroups} (they will be considered in Example \ref{cee+}).

Let $n\ge 2$ and  $(a_1,a_2,\ldots,a_n)$ be a sequence of relatively prime positive integers. For
every $k\in \{1,\ldots,n\}$, let $d_k=\gcd(a_1,\ldots,a_k)$. 
For $k\in\{2,\ldots,n\}$, let $c_k=d_{k-1}/d_k$.
Let $S_k$ be the semigroup generated by $a_1,\ldots,a_k$. We say that the
sequence $(a_1,a_2,\ldots,a_n)$ is \emph{smooth} if $c_ka_k\in S_{k-1}$ for every 
$k=2,\ldots,n$.

Observe that a numerical semigroup $S$ is generated by a smooth sequence if and only if $S$ is free. Also $c_k a_k\in S_{k-1}$ is equivalent to $\frac{a_k}{d_k}\in \frac{1}{d_{k-1}}S_{k-1}$ (and $\frac{1}{d_{k-1}}S_{k-1}$ is a numerical semigroup). 
Notice that $S_k=S_{k-1}+a_k\mathbb N$. With the notation of gluing, we have $\frac{1}{d_k}S_k=c_k\left(\frac{1}{d_{k-1}}S_{k-1}\right)+_{c_k\frac{a_k}{d_k}} \frac{a_k}{d_k}\mathbb N$.  By using \eqref{pol-ci}, we recover the following result.

\begin{Lem}[{Leher \cite[Corollary 8]{Leher}}] 
\label{duvel}
Let $n\ge 2$ and  $(a_1,a_2,\ldots,a_n)$ be a smooth sequence. Let $S$ be the numerical semigroup generated by $a_1,\ldots,a_n$.
We have
$$\mathrm P_S(x)=(1-x)\prod_{i=2}^n (1-x^{c_ia_i})\prod_{i=1}^n(1-x^{a_i})^{-1},$$
which factorizes as
\begin{equation}
\label{beeplus}
\mathrm P_S=\Phi_1 \prod_{d\mid a_1}\Phi_d^{-1}\prod_{i=2}^n\prod_{d\mid c_ia_i,~d\nmid a_i}\Phi_d.
\end{equation}
\end{Lem}

\begin{cor} \label{duvel2} 
 Let $S$ be the numerical semigroup generated by the smooth sequence $(a_1,\ldots,a_n)$
 with $n\ge 2$. Then
\begin{enumerate}[a)]
\item $\mathrm F(S)=\sum_{i=2}^n c_ia_i-\sum_{i=1}^{n}a_i$ (this formula can also be derived from \cite{Delorme} or \cite{H-K}).
\item $S$ is symmetric.
\item $S$ is cyclotomic.
\end{enumerate}
\end{cor}

\begin{Example}[Binomial semigroups]\label{cee+}
Consider $B_m(a,b):=\langle a^m,ba^{m-1},\ldots,b^{m-1}a,b^m\rangle$, where $a,b>1$ are
relatively prime. Putting $a_k=a^{m-k}b^k$, $k\in\{0,\ldots,m\}$, we see that the sequence
$(a_0,\ldots, a_m)$ is smooth (with $c_k=a$ for $k\in\{1,\ldots,m\}$ and $c_ka_k=ba_{k-1}\in \langle a_0,\ldots,a_{k-1}\rangle$).
By Corollary \ref{duvel2} it follows that
$$\mathrm F(B_m(a,b))=\sum_{k=1}^m a^{m+1-k}b^k-\sum_{k=0}^m a^{m-k}b^k.$$
Further, we have
$$\mathrm P_{B_n(a,b)}(x)=(1-x)\prod_{k=1}^m (1-x^{a^{m+1-k}b^k})\prod_{k=0}^m (1-x^{a^{m-k}b^k})^{-1}.$$
In particular, let $ B=B_n(p,q) $ be a binomial numerical semigroup with $ p $ and $ q $ different primes. From \eqref{beeplus} we infer that
\begin{equation}
\label{depths}
\mathrm P_B  =  \Phi_1 (\Phi_1\Phi_p \cdots \Phi_{p^n})^{-1}\prod_{k=1}^n \prod_{j=0}^k \Phi_{p^{n+1-k}q^j}
 =  \prod_{k=1}^n \prod_{j=1}^k \Phi_{p^{n+1-k}q^j}=\prod_{l=2}^{n+1}\prod_{\substack{i+j=l\\1\le i,j\le l}}\Phi_{p^iq^j}.
\end{equation}

\end{Example}

\section{Cyclotomic exponents and a first step in proving the conjecture }\label{cyclo}
The reader might wonder whether the expression in the right-hand side of (\ref{pol-ci}) is unique. It
is easy to see the answer is yes and indeed a little more can be shown, see Moree \cite[Lemma 1]{M2000}.
\begin{Lem} 
Let $f(x)=1+a_1x+\dots+a_dx^d\in \mathbb Z[x]$ be a polynomial of degree $d$ (hence $a_d\ne 0$).
Let $\alpha_1,\ldots,\alpha_d$ be its roots. Put $s_f(k)=\alpha_1^{-k}+\dots+\alpha_d^{-k}$. Then
the numbers $s_f(k)$ are integers and satisfy the recursion
$$s_f(k)+a_1s_f(k-1)+\dots+a_{k-1}s_f(1)+ka_k=0,$$
with $a_k=0$ for every $k>d$. Put 
$$b_f(k)=\frac{1}{k}\sum_{d\mid k}s_f(d)\mu\left(\frac{k}{d}\right).$$
Then $b_f(k)$ is an integer. Moreover, we have the formal identity
$$1+a_1x+\dots+a_dx^d=\prod_{j=1}^{\infty}(1-x^j)^{b_f(j)}.$$
\end{Lem}
It is a consequence of this lemma that
given a numerical semigroup $S$, there are unique
integers $\mathrm e_1,\mathrm e_2,\ldots$ such that
\begin{equation}
\label{cycident}
\mathrm P_S(x)=\prod_{j=1}^{\infty}(1-x^j)^{\mathrm e_j}.
\end{equation}
The sequence ${\bf e}=\{\mathrm e_1,\mathrm e_2,\ldots\}$ we call the {\it cyclotomic exponent sequence} of $S$.
\begin{Problem}
Relate the properties of $S$ to its cyclotomic exponent sequence.
\end{Problem}
By Lemma \ref{trieviaal} we have $\mathrm e_1=0$ if $S=\langle 1\rangle$ and $\mathrm e_1=1$ otherwise. 
\begin{Remark}
The identity (\ref{cycident}) only holds for $|x|<\delta<1$, where $\delta$ is 
postive and easily related
to the roots of $f,$ see \cite{M2000} for details.
\end{Remark}

\begin{Lem}\label{finitesequence}
A numerical semigroup $S$ has a cyclotomic exponent sequence
with finitely many nonzero terms if and only if $S$ is 
a cyclotomic numerical semigroup.
\end{Lem}
\begin{proof}
"$\Rightarrow$". We can write
$\mathrm P_S(x)=\prod_{j=1}^{k}(1-x^j)^{\mathrm e_j}$ for some $k,$ hence $\mathrm P_S(x)$ 
has only roots of unity as zeros and so $S$ is a cyclotomic numerical
semigroup.\\
"$\Leftarrow$". By Definition \ref{cy}, the fact that $\mathrm P_S(1)\ne 0$ and formula (\ref{phi-signchange}) we
infer that $\mathrm P_S(x)=\prod_{j=1}^{r}(1-x^j)^{f_j}$, with the $f_j$ integers. By the uniqueness of the
cyclotomic exponents it now follows that $f_j=\mathrm e_j$ and so $\mathrm e_j=0$ for all $j$ large enough. 
\end{proof}
Lemma \ref{kronie} makes the above result effective using the arithmetic function
$a(n)=\max\{m \colon\varphi(m)\le n\}$. Using the estimate $\varphi(n)\ge (n/3)^{2/3}$ (see, e.g., 
\cite{effective}) we
see that we can write $a(n)=\max\{1\le m\le 3n^{3/2} \colon \varphi(m)\le n\}$, giving an 
algorithm to determine $a(n)$.
\begin{Lem}
\label{kronie}
Let $S$ be a numerical semigroup. Put $r=a(\mathrm F(s)+1)$.
Then $S$ is cyclotomic if and only if
$\mathrm P_S(x)=\prod_{j=1}^{r}(1-x^j)^{\mathrm e_j}$, with
$\mathrm e_j$ the cyclotomic exponents of $\mathrm P_S$.
\end{Lem}
\begin{proof}  
By Lemma \ref{finitesequence} we have that $S$ is cyclotomic if and only if
 $\mathrm P_S(x)=\prod_{j=1}^{M}(1-x^j)^{\mathrm e_j}$ for some integer $M$, so that 
 $\mathrm e_j=0$ for all $j>M$. 
 Now
 $\Phi_M(x)$ divides $1-x^M$ and no $1-x^j$ with $1\le j<M,$ and so $\Phi_m(x)\mid\mathrm P_S(x)$. 
 If $M>a(F(s)+1)$ it would follow
 that $\varphi(M)>F(s)+1$, and thus the product would have degree $>\mathrm F(S)+1$, whereas
 $\mathrm P_S(x)$ has degree $\mathrm F(S)+1$. This contradiction shows that $M\le r$. \end{proof}

The proof of the latter lemma is easily adapted to
show the correctness of the following algorithm which determines whether a monic polynomial $f(x)\in \mathbb Z[x],$ with
$f(0)\ne 0,$ is Kronecker or not.
\begin{algorithm}
Let $f(x)\in \mathbb Z[x]$ with
$f(0)\ne 0$ be a monic polynomial of degree $d$.
We are going to write $f(x)$ as $h(x)\prod_{1\le d\le M}\Phi_d(x)^{g_d}$, with 
$h(x)$ coprime to the product $\prod_{1\le d\le M}\Phi_d(x)^{g_d}$ and $M\le a(d)$. 
We determine
the $\gcd$ of $f(x)$ and $\Phi_1(x)$. As long as this $\gcd$ is $\Phi_1(x)$ we divide
$\Phi_1(x)$ out and continue until the $\gcd$ is $1$. 
We keep track of the number of divisions and in this way we have determined $g_1$.
We proceed with taking the $\gcd$ with $\Phi_2(x)$ and repeat
the process. In this way we determine the $\prod_{1\le d\le M}\Phi_d(x)^{g_d}$
and $h(x)$. Then $f(x)$ is Kronecker if and only if $h(x)=1$.
\end{algorithm}
\subsection{Cyclotomic exponents and Betti numbers}\label{Beatrice}
Write $ S=\langle n_1,n_2,\ldots,n_e\rangle $, with $e=\mathrm e(S)$ and $0<n_1<\cdots<n_e$. Note that \[(1-x^{n_1})(1-x^{n_2})\cdots(1-x^{n_e})=\sum_{j_1=0}^{1}
\sum_{j_2=0}^{1}\cdots\sum_{j_e=0}^{1}(-1)^{j_1+j_2+\cdots+j_e}x^{j_1n_1+j_2n_2+\cdots+j_en_e}.\]
We can thus write 
\begin{equation}
\label{looong}
\mathrm{P}_S(x)=\frac{1-x}{(1-x^{n_1})\cdots(1-x^{n_e})}\left(  \sum_{j_1=0}^{1}\sum_{j_2=0}^{1}\cdots\sum_{j_e=0}^{1}(-1)^{j_1+j_2+\cdots+j_e}x^{j_1n_1+j_2a_2+\cdots+j_en_e+S}\right),
\end{equation}
where $m+S:=\left\lbrace m+s: s\in S\right\rbrace$.
On recalling the definition of $ \chi_S(n) $ given in Section \ref{glueglue}, we can rewrite \eqref{looong} as 
\begin{equation}
\label{polychi}
\mathrm P_S(x)=\frac{1-x}{(1-x^{n_1})\cdots(1-x^{n_e})}\sum_n \chi_S(n)x^n.
\end{equation}
Note that $ \sum_n\chi_S(n)x^n $ is a polynomial since every 
$n>\mathrm F(S)+n_1+\cdots+n_e$ can be written as $ \sum a_in_i $ with $a_i\ge 1$ for $ 1\le i\le e $ and hence $\chi_S(n)=0$; this recovers the formula given in \cite{s-w}. Alternatively, this can be seen by noting that $\sum_n \chi_S(n)x^n$ is the product
of the polynomials $\mathrm P_S(x)$ and $(1-x^{n_1})\cdots(1-x^{n_e})/(1-x)$.
\par As a first step in proving Conjecture \ref{conjecture}, the following can be shown. Let $\mu=\min\left\lbrace n>1 : \chi_S(n)\ne 0\right\rbrace$ and let ${\mathfrak d}(n)$ be the \textit{denumerant} of $n$, that is, 
$${\mathfrak d}(n)=\#\left\{(a_1,\ldots,a_e)\in {\mathbb N}^e:\sum a_in_i=n\right\}.$$ 
\begin{Lem}
\label{ceva}
Let $S=\langle n_1,\ldots,n_e \rangle$ be a minimally generated cyclotomic numerical semigroup such that $ n_1<\cdots<n_e $ and let ${\bf e}=\{\mathrm e_1,\mathrm e_2,\ldots\}$ be its cyclotomic exponent sequence.
\begin{enumerate}[a)]
\item We have $\mathrm e_1=1 $. If $ \mu>n_e $, then $ \mathrm e_{n_1}=\mathrm e_{n_2}=\cdots =\mathrm e_{n_e}=-1 $. Further $ \mathrm e_{\mu}=-\chi_S(\mu) $. If $ 1\leq j\leq \mu $ and $ j\notin\left\lbrace 1,n_1,\ldots,n_e,\mu\right\rbrace  $, then $\mathrm e_j=0$.
\item If, in addition, there is some $s\in S$, $s\leq n_e$ with $\mathfrak d(s)\geq 2$, then $\mathrm e_j\geq 0$ for all $j>n_e$.
\end{enumerate}
\end{Lem}
\begin{proof}$~$\\
a) Since $\chi_S(0)=1$ we have $\sum_n \chi_S(n)x^n=1+\chi_S(\mu)x^{\mu}+\cdots$ and the result follows.\\
b) If there exist both positive $\mathrm e_j$ and negative $\mathrm e_i$ with $i, j>n_e,$ then we can write $$\mathrm P_S(x)=\frac{1-x}{\prod_{i=1}^e (1-x^{n_i})}\cdot \frac{(1-x^{j_1})^{\mathrm e_{j_1}}(1-x^{j_2})^{\mathrm e_{j_2}}\cdots}{(1-x^{i_1})^{\mathrm e_{i_1}}(1-x^{i_2})^{\mathrm e_{i_2}}\cdots},$$ where both products in the numerator and denominator are finite, $n_e<j_1<j_2<\cdots,$ $n_e<i_1<i_2<\cdots,$ and $\mathrm e_{i_k},\mathrm e_{j_l}> 0$.  Now, $(1-x^{i_1})^{-\mathrm e_{i_1}}(1-x^{i_2})^{-\mathrm e_{i_2}}\cdots=(1+x^{i_1}+x^{2i_1}+\cdots)^{\mathrm e_{i_1}}\cdots=1+\beta x^b+O(x^{b+1})$ is an infinite sum, with $b=i_1>n_e,\beta= \mathrm e_{i_1}$, and $(1-x^{j_1})^{\mathrm e_{j_1}}(1-x^{j_2})^{\mathrm e_{j_2}}\cdots =1-\alpha x^a+O(x^{a+1})$ is a finite sum, with $a=j_1>n_e,\alpha= \mathrm e_{j_1}$. Hence, using the easy fact that
\begin{equation}
\label{airways}
\prod_{i=1}^e (1-x^{n_i})^{-1}=\prod_{i=1}^e\sum_{j=0}^{\infty}x^{jn_i}=\sum_{s\in S}\mathfrak d(s)x^s, 
\end{equation} we get
\begin{equation}\label{eq}
\mathrm H_S(x)=\left( \sum_{s\in S}\mathfrak d(s)x^s\right) (1-\alpha x^a+O(x^{a+1}))(1+\beta x^b+O(x^{b+1})).  
\end{equation} Note that $a\neq b,$ hence we distinguish two cases:\\
\noindent{\bf Case 1:} $a>b$. Then $(1-\alpha x^a+O(x^{a+1}))(1+\beta x^b+O(x^{b+1}))=1+\beta x^b+O(x^{b+1}),$ so that, after multiplication by $\sum \mathfrak d(s)x^s$, the power $x^b$ does not get reduced and appears, with some coefficient, in the Hilbert series, hence $b\in S$. But then this coefficient will be at least $\mathfrak d(b)+\beta\geq 1+1= 2$, contradiction.

\noindent {\bf Case 2:} $a<b$. Then $(1-\alpha x^a+O(x^{a+1}))(1+\beta x^b+O(x^{b+1}))=1-\alpha x^a+O(x^{a+1}),$ hence $\left( \sum\mathfrak d(s)x^s\right) (1-\alpha x^a+O(x^{a+1}))(1+\beta x^b+O(x^{b+1}))=\sum_{s\in S,s<a} \mathfrak d(s)x^{s}+\cdots$ and, by assumption, there is some $s\leq n_e<a$ with $\mathfrak d(s)\geq 2$, leading to a coefficient greater than 1 in the Hilbert series, contradiction. \\
\indent Note that having $\mathrm e_i\leq 0$ for all $i>n_e$ is impossible, as we would then 
get $\mathrm H_S(x)=\left( \sum\mathfrak d(s)x^s\right)\left(\sum_j r_jx^j\right ) $ 
with $r_j\ge 0$. On expanding this, one can find coefficients larger than 1 in $ \mathrm{H}_S(x) $. 
Therefore we can only have $\mathrm e_j\geq 0$ for all $j>n_e$.
\end{proof}

It follows that we can express the Hilbert series of a numerical semigroup $ S $ satisfying the conditions of Lemma \ref{ceva} as 
\begin{equation}
\label{hilbertdecompose}
\mathrm H_S(x)=\frac{(1-x^{d_1})^{\mathrm e_1}\cdots(1-x^{d_k})^{\mathrm e_k}}{(1-x^{n_1})\cdots(1-x^{n_e})},
\end{equation}
where $n_e<d_1<d_2<\cdots<d_k$ and $\mathrm e_i\ge 1$, $i=1,\ldots,k$ (with $\sum_{i=1}^k \mathrm e_i=m-1$, 
which follows on noting that $\mathrm H_S(x)$ must have a simple pole at $x=1$).
The conditions of 
Lemma \ref{ceva} are rather restrictive. However, solely from a factorization such as \eqref{hilbertdecompose}, it is easy to prove the following.
\begin{Lem}
Let $S=\langle n_1,\ldots,n_e \rangle$ be a minimally generated numerical semigroup such that \eqref{hilbertdecompose} holds. Then $d_i\in S$ for all $i=1,\ldots,k$ and $d_1=\min\left\lbrace s:  s\in\mathrm{Betti}(S)\right\rbrace $.
\end{Lem}   
\begin{proof} Rewrite \eqref{hilbertdecompose} as 
\begin{equation}
\label{betti}
(1-\mathrm e_1x^{d_1}+\cdots)\cdots(1-\mathrm e_kx^{d_k}+\cdots)=\sum_{s\in S} x^s\prod_{i=1}^{e}(1-x^{n_i}).
\end{equation}
The right-hand side of \eqref{betti} is of the form $\sum_{s\in S}r(s)x^s,$ for some $ r(s)\in\mathbb Z, $ while the left equals $1-\mathrm e_1x^{d_1}+O(x^{d_1+1})$. Then $ d_1\in S. $ Next, when expanding the left-hand side, the power $ x^{d_2} $ either gets cancelled by a power $ x^{\alpha d_1}, $ for some $ \alpha\in\mathbb N, $ or appears in the sum with a nonzero coefficient. Either case, $d_2\in S$ and we can repeat the same argument to show that $ d_3,\ldots,d_k\in S$. Combining \eqref{airways} and \eqref{betti} yields 
\[\mathrm H_S(x)=(1-\mathrm e_1x^{d_1}+\cdots) \sum\mathfrak d(s)x^s=\sum_{s\in S,s<d_1}\mathfrak d(s)x^s+(\mathfrak d(d_1)-\mathrm e_1)x^{d_1}+\cdots,\]
 hence $ d_1  $ is the first element $ s\in S $ with $ \mathfrak d(s)\ge 2 .$ We prove that this implies $d_1=\min\left\lbrace s:  s\in\mathrm{Betti}(S)\right\rbrace .$ Note that, by definition, $ \mathfrak d(s)\ge 2 $ for any $ s\in\mathrm{Betti}(S). $ Therefore it suffices to prove that $ d_1\in\mathrm{Betti}(S). $ 
 \par Let $ d_1=a_1n_1+\cdots+a_en_e =b_1n_1+\cdots+b_en_e $ be two different representations of $ d_1 $ in terms of the generators, with $ a_i,b_i\in\mathbb N. $ If there is $ 1\le i\le e $ such that $ a_i,b_i>0, $ then $ d_1-n_i\in S $ and $ \mathfrak d(d_1-n_i)\ge 2, $ contradiction. But this implies that the underlying graph of $ \Delta_{d_1} $ is disconnected, i.e., $ d_1\in\mathrm{Betti}(S). $ Indeed, take any two distinct representations $ d_1=a_1n_{i_1}+\cdots+a_kn_{i_k}=b_1n_{j_1}+\cdots+b_ln_{j_l}, $ where $ a_i,b_j>0, $ $ k,l\ge1 $ and $ \left\lbrace i_1,\ldots,i_k\right\rbrace \cap\left\lbrace j_1,\ldots,j_l\right\rbrace=\emptyset. $ Then there can be no edge between $ n_{i_{\alpha}} $ and $ n_{j_{\beta}} $ in the underlying graph of $ \Delta_{d_1} .$ Otherwise, if say, $ n_{i_1}n_{j_1} $ is an edge, then $ n:=d_1-n_{i_1}-n_{j_1}\in S $ and thus $ d_1-n_{i_1}=(a_1-1)n_{i_1}+a_2n_{i_2}+\cdots+a_kn_{{i_k}} =n+n_{j_1}$ admits at least two distinct representations, contradiction. Hence the vertices $  n_{i_{\alpha}} $ and respectively $  n_{j_{\beta}} $ lie in distinct connected components. Consequently, the underlying graph of $ \Delta_{d_1} $ is disconnected, that is, $ d_1\in\mathrm{Betti}(S). $  
\end{proof}

\section{Polynomially related numerical semigroups}

We say that a numerical semigroup $S$ is \emph{polynomially related
to} the numerical semigroup $T,$ and denote this by $ S\le_P T, $ if there exist $f(x)\in {\mathbb Z}[x]$ 
and an integer $w\ge 1$ such that
\begin{equation}
\label{StoT}
\mathrm H_S(x^w)f(x)=\mathrm H_T(x).
\end{equation}
From \eqref{StoT} we infer that 
\begin{equation}
\label{StoT2}
\mathrm P_S(x^w)f(x)=\mathrm P_T(x)(1+x+\dots+x^{w-1}).
\end{equation}
Note that \eqref{StoT} and \eqref{StoT2} are equivalent formulations of $S$
being polynomially related to $T$.
\begin{Example}
Put $S_1=\langle p,q\rangle$ and $S_3=\langle p^3,q\rangle$.  By \eqref{Pa1a2} we
have $\Phi_{pq}\Phi_{p^2q}\Phi_{p^3q}=\mathrm P_{S_3}$. Recall that $\mathrm P_{S_1}=\Phi_{pq}$.
We have
$$\mathrm P_{S_1}(x)\Phi_{p^2q}(x)\Phi_{p^3q}(x)=\mathrm P_{S_3}(x),~\mathrm P_{S_1}(x^p)\Phi_{pq}(x)\Phi_{p^3q}(x)=\mathrm P_{S_3}(x),~
\mathrm P_{S_1}(x^{p^2})\Phi_{pq}(x)\Phi_{p^2q}(x)=\mathrm P_{S_3}(x),$$
giving three different polynomial relations between $S_1$ and $S_3$.
\end{Example}
\begin{Lem}
Being polynomially related defines a partial order
on the set of numerical semigroups.
\end{Lem}
\begin{proof} Obviously a numerical semigroup is polynomially related to itself. 
Further, being polynomially related is clearly transitive. Using part d) of Lemma \ref{feen} we see
that $\mathrm F(S)<\mathrm F(T)$ unless $S=T$. This implies that being polynomially related defines an antisymmetric binary
relation on the set of numerical semigroups.
\end{proof}
\begin{Problem}
Find necessary and sufficient conditions for $S$ to be polynomially 
related to $T$.
\end{Problem}
In proving the following result we make repeatedly use of the fact
that $\mathrm  P_S(1)=1$ and $\mathrm  P_S'(1)=\mathrm g(S)$ (see Lemma \ref{gaps}.)
\begin{Lem}
\label{feen}
Suppose that $\mathrm H_S(x^w)f(x)=\mathrm H_T(x)$ holds with $S,T$ numerical semigroups. Then
\begin{enumerate}[a)]
\item $f(0)=1$.
\item $f(1)=w$.
\item $f'(1)=w(\mathrm g(T)-w\mathrm g(S)+(w-1)/2)$.
\item $\mathrm F(T)=w\mathrm F(S)+\mathrm{deg}(f)$.
\item If $w$ is even, then $f(-1)=0$.
\item If $w$ is odd, then $f(-1)=\mathrm P_T(-1)/\mathrm P_S(-1)$. 
\item If $T$ is cyclotomic, then so is $S$.
\item If $S$ is cyclotomic, then $T$ is cyclotomic if and only if $f$
is Kronecker.
\end{enumerate}
\end{Lem}
\begin{proof}$~$\\ 
a) We have $ \mathrm P_S(0)=\mathrm P_T(0)=1$.\\
b) On substituting $x=1$ in the identity \eqref{StoT2} and noting that $\mathrm P_S(1)=\mathrm P_T(1)=1$, we obtain $f(1)=w$.\\
c) The identity
\eqref{StoT2} yields (on differentiating both sides) that 
$$\mathrm P_S'(x^w)wx^{w-1}f(x)+\mathrm \mathrm \mathrm P_S(x^w)f'(x)=\mathrm P_T'(x)(1+x+\dots+x^{w-1})+\mathrm P_T(x)\sum_{j=0}^{w-2}(j+1)x^j.$$ The claim now 
easily follows
on setting $x=1$ and invoking part b).\\
d) Use that deg$( \mathrm P_S)=\mathrm F(S)+1$.\\
e)+f)  Note that $\mathrm P_S(-1)\ne 0$ and 
substitute $x=-1$ in (\ref{StoT2}).\\
g)+h) Obvious.
 \end{proof}

The next result gives more specific information about $f$ in
case $f$ has nonnegative coefficients only.
\begin{Lem}\label{lem-pol-rel-ei-iff}
Suppose that $S$ and $T$ are numerical semigroups. 
Then $\mathrm H_S(x^w)f(x)=\mathrm H_T(x)$ for some integer $w\ge 1$ and $f\in {\mathbb N}[x]$ if and only if
there are $0=e_1<e_2<\dots<e_w$ such that
$f(x)=\sum_{i=1}^w x^{e_i}$ and every $t\in T$ can be written in a unique way as
$$t=e_i+s\cdot w,~1\le i\le w,~s\in S.$$
\end{Lem}
\begin{proof}"$\Rightarrow$". If $f$ were to have a coefficient greater than $1$, this would lead to a coefficient greater than $1$ in $\mathrm H_T$, which
is not possible. By Lemma \ref{feen} we have $f(0)=1$ and $f(1)=w$, and hence it follows that
$f(x)=\sum_{i=1}^w x^{e_i}$ with $0=e_1<\dots<e_w$.
The identity $\sum_{i=1}^w x^{e_i}\sum_{s\in S}x^{sw}=\mathrm H_T(x)$ yields that every element $t\in T$ can
be written as $t=e_i+s\cdot w$, with $1\le i\le w$ and $s\in S$. Since every nonzero coefficient of $\mathrm H_T$ is $1$, this
writing way of $t$ must be unique.\\ 
"$\Leftarrow$". Obvious. 
\end{proof}

Compare the expression of $t$ in the above lemma with \cite[Lemma 2.6]{roos}.

\begin{Remark}
By Lemma \ref{feen} we have $\sum_{i=1}^w e_i=w(\mathrm g(T)-w\mathrm g(S)+(w-1)/2)$.
\end{Remark}

\begin{cor}$~$ \label{th-pol-related-emb-2}
\begin{enumerate}[a)]
\item We have $\langle p^a,q^b\rangle \le_P \langle p^m,q^n\rangle$ if $1\le a\le m$ and $1\le b\le n$.
\item We have $\langle p^a,q^b\rangle \le_P B_n(p,q)$ if $a,b\ge 1$ and $2\le a+b\le n+1$.
\item Let $V$ be a numerical semigroup generated by $\{n_1,\ldots,n_k\}$. Let
$d=\gcd(n_1,\ldots,n_{k-1})$ and set $U=S(n_1/d,\ldots,n_{k-1}/d,n_k)$. The numerical semigroup $U$ is polynomially related to $V$.
\end{enumerate}
\end{cor}
\begin{proof}$~$\\
{\rm a)} This is a consequence of the identity
\begin{equation}\label{ex}
\mathrm P_{\langle p^m,q^n\rangle}(x)=\prod_{\substack{1\le \alpha \le m\\ 1\le \beta \le n}}\Phi_{p^{\alpha}q^{\beta}}(x),
\end{equation}
which is a consequence of \eqref{Pa1a2}.\\
{\rm b)} Results on comparing \eqref{ex} with the factorization of $\mathrm P_B$ given in 
Example \ref{cee+}.\\ 
{\rm c)} It is easy to see (cf. \cite[Lemma 2.16]{roos}) that $\Ap(V;n_k)=d\Ap(U;n_k)$. By using this identity and \eqref{oudedoos} we derive 
\[\mathrm H_U(x^d)\left(\frac{1-x^{n_kd}}{1-x^{n_k}}\right)=\mathrm H_V(x).\qedhere\]
\end{proof}

\subsection{An application}
We will use our insights into polynomially related 
numerical semigroups to establish the following result.
\begin{Thm} 
\label{examp}
Let $p\neq q$ be primes and $m,n$ positive integers.
The quotient 
$$Q(x):=\mathrm P_{\langle p^m,q^n\rangle}(x)/\Phi_{p^mq^n}(x)$$ is  monic, is in ${\mathbb Z}[x]$ and has constant
coefficient $1$. Its nonzero 
coefficients alternate
between $1$ and $-1$.
\end{Thm} 
\begin{proof} On using that $\mathrm P_S(x)=(1-x)\mathrm H_S(x)$ and the identity \eqref{Pa1a2}, we infer
that  
\begin{equation}
\label{blubber}
\mathrm H_{\langle p^m,q^n\rangle}(x)=\mathrm H_{\langle p,q\rangle}(x^{p^{m-1}q^{n-1}})\sum_{j=0}^{q^{n-1}-1}x^{jp^m}\sum_{k=0}^{p^{m-1}-1}x^{kq^n}.
\end{equation}
The identity \eqref{ex} yields that $Q(x)$ is a polynomial in ${\mathbb Z}[x]$. On noticing that
$$\mathrm P_{\langle p,q\rangle}(x^{p^{m-1}q^{n-1}})=\Phi_{p^mq^n}(x)$$ we obtain from \eqref{blubber} that 
$$Q(x)=\frac{1-x}{ 1-x^{p^{m-1}q^{n-1}}}\sum_{j=0}^{q^{n-1}-1}x^{jp^m}\sum_{k=0}^{p^{m-1}-1}x^{kq^n}.$$
The set $$\{ \alpha p^m + \beta q^n \colon 0\le \alpha \le q^{n-1}-1,~0\le \beta \le p^{m-1}-1\}$$
forms a complete residue system modulo $p^{m-1}q^{n-1}$ and it follows
that around $x=0$ we can
write $Q(x)=(1-x)\sum_{s\in S'}x^s$ for some
set $S'$ containing zero and all large
enough integers. {}From this it follows that 
$Q(x)$ is a monic polynomial and that the nonzero coefficients of $Q(x)$ alternate between $1$ and $-1$. 
\end{proof}

\begin{Remark}
An alternative, much more conceptual proof 
of the identity \eqref{blubber}  is obtained on using the following
lemma; one notes that on writing down the Hilbert series
for both sides of \eqref{elf}, we obtain the identity \eqref{blubber}.
\end{Remark}

\begin{Lem}
\label{lemma}
Let $T=\langle p^m,q^n \rangle$ and $S=\langle p,q \rangle$. Every element of $T$ can be uniquely
written as
\begin{equation}
\label{elf}
t=\alpha p^m + \beta q^n + s p^{m-1}q^{n-1},0\le \alpha \le q^{n-1}-1,0\le \beta \le p^{m-1}-1,s\in S.
\end{equation}
\end{Lem} 
\begin{proof} Suppose that $t\in T$. Then
\begin{equation}
\label{sterr}
t=ap^m+bq^n=(q^{n-1}a_1+\alpha)p^m+(p^{m-1}b_1+\beta)q^n, 
\end{equation}
with
$0\le \alpha \le q^{n-1}-1$ and $0\le \beta \le p^{m-1}-1$.
Put $s=a_1p+b_1q$. Clearly $s\in S$. {}From \eqref{sterr} we then infer
that $t=\alpha p^m + \beta q^n + s p^{m-1}q^{n-1}$, as required.
The congruence class of $t$ modulo $p^{m-1}q^{n-1}$ determines $\alpha$ and $\beta$ uniquely.
Since $\alpha$ and $\beta$ are determined uniquely, so is $s$. 
\end{proof}

\noindent Theorem \ref{examp} can be alternatively proven on invoking the following more general result
together with Lemma \ref{lemma}.
\begin{Thm}\label{thm-pol-rel-ei}
Suppose that $S$ and $T$ are numerical semigroups with $\mathrm H_S(x^w)f(x)=\mathrm H_T(x)$ for some $w\ge 1$ and
$f\in {\mathbb N}[x]$.
Put $Q(x)=\mathrm P_T(x)/ \mathrm P_S(x^w)$. Then $Q(0)=1$, $Q(x)$ is a 
monic polynomial and its nonzero coefficients alternate between
$1$ and $-1$.
\end{Thm}
\begin{proof} By Lemma \ref{lem-pol-rel-ei-iff} we can write $f(x)=\sum_{i=1}^w x^{e_i}$. Since $T$ contains
all integers sufficiently large, it follows
that $e_1,\ldots,e_w$ form a complete residue system modulo $w$.
By \eqref{StoT2} we see that
$$Q(x)=\frac{f(x)}{ 1-x^w}(1-x).$$
Around $x=0$ we have $f(x)/(1-x^w)=\sum_{z\in Z}x^z$ for some infinite set of integers $Z$. Since 
$e_1,\ldots,e_w$ form a complete residue system modulo $w$, it follows that all integers large enough are in $Z$.
{}From this we then infer that $Q(x)$ is a monic polynomial. 
Note that $Q(0)=f(0)=1$ by Lemma \ref{feen} and so $0\in Z$.
For any set $Z'\subseteq \mathbb N$ containing 0, the nonzero coefficients in
$(1-x)\sum_{z\in Z'}x^z$ alternate between $1$ and $-1$. 
\end{proof}

\section{Cyclotomic numerical semigroups of prescribed height and depth}
It follows from Lemma \ref{begin} and the identity 
\eqref{uitelkaar} that if $S$ is a cyclotomic numerical semigroup, then
$\mathrm P_S(x)\mid (x^m-1)^e$ for some integers $m$ and $\mathrm e$.\\
\indent We say that a numerical semigroup
$S$ is cyclotomic of \emph{depth} $d$ and  \emph{height} $h$ if $\mathrm P_S(x)\mid (x^d-1)^h$, where both $d$ and $h$
are chosen minimally, that is, $\mathrm P_S(x)$ does not divide $(x^n-1)^{h-1}$ for any $n$ and it does
not divide $(x^{d_1}-1)^h$ for any $d_1<d$.\\
\indent On noting that $\Phi_m(x)\mid (x^n-1)$ if and only if $m\mid n$ one arrives at the following conclusion.
\begin{Lem}
\label{floepie}
Suppose that $S$ is a cyclotomic numerical semigroup with $\mathrm P_S$
factorizing as in \eqref{basaal}, namely
$$\mathrm P_S(x)=\prod_{i=1}^s\Phi_{d_i}(x)^{e_{d_i}},$$ where $ d_i $ and $ e_{d_i} $ are positive integers.
Then $S$ is
of depth ${\rm lcm}(d_1,d_2,\ldots,d_s)$ and 
of height $\max\{e_1,\ldots,e_s\}$.
\end{Lem} 
\begin{Example} Consider the binomial semigroup $ B=B_n(p,q) $ defined in Example \ref{cee+}. 
By Lemma \ref{floepie} and on recalling the factorization \eqref{depths}, we see that $B$ is of depth $d=p^{n+1}q^{n+1}$ and of height $h=1$. 
\end{Example}
\begin{Problem}
\label{dh}
Classify all cyclotomic numerical semigroups having a prescribed depth and height.
\end{Problem}
In the other direction we might ask 
for divisors of $x^n-1$ that
are semigroup polynomials. 
Various authors studied the coefficients of divisors of $x^n-1$ \cite{DM, Kaplan, PR, RWW, SW}.
By Corollary \ref{corrie} we know that if a divisor $f(x)$ of $x^n-1$ is of 
the form $\mathrm P_S(x),$ then its
nonzero coefficients alternate between $1$ and $-1$.

We start with considering Problem \ref{dh} for height 
$h=1$. We will need the following trivial observation.
\begin{Lem}
\label{trieviaal}
If $S\ne \langle 1\rangle$, then $ \mathrm P_S(x)\equiv 1-x \pmod{x^2}$.
\end{Lem}
\begin{proof} If $S\ne \langle 1\rangle$, then $0\in S$ and $1\not\in S$ and
hence $\sum_{s\in S}x^s\equiv 1 \pmod{x^2}$. \end{proof}

\begin{Thm}
Let $p,q$ and $r$ be pairwise distinct primes. Suppose $S$ is cyclotomic of
depth $d=pqr$ and height $h=1$. Then $S=\langle pr,q\rangle$ up to a cyclic permutation of $p,q,r.$
\end{Thm}
\begin{proof} Suppose that $ \mathrm P_S(x)\mid x^{pqr}-1$ for some $S$. Then by \eqref{uitelkaar} and Lemma \ref{begin}
we have $\mathrm  P_S=\Phi_{pq}^{k_1}\Phi_{qr}^{k_2}\Phi_{pr}^{k_3}\Phi_{pqr}^{k_4}$ with $0\le k_i\le 1$.
Since the problem is symmetric in $p,q$ and $r$, we may assume without loss of generality  that $k_1\ge k_2\ge k_3$.
Note that, modulo $x^2$, $f(x)=1+(k_4-k_1-k_2-k_3)x$. On invoking Lemma \ref{trieviaal} we now deduce that $(k_1,k_2,k_3,k_4)\in \{(1,0,0,0),(1,1,0,1)\}$. The first case we can exclude,
as this leads to a depth $d=pq$. By \eqref{Pa1a2} we have  $\Phi_{pq}\Phi_{qr}\Phi_{pqr}=\mathrm P_{\langle pr,q\rangle}$. 
\end{proof}
\begin{Thm}
\label{pnq}
Suppose $T$ is a cyclotomic numerical semigroup of depth $d=p^nq$ and height $h=1$. Then
$T=\langle p^n,q\rangle$.
\end{Thm}
The proof makes use of the following lemma.
\begin{Lem}
\label{rijtje}
Let $k\ge 1$ be an integer, $0\le e_i\le 1$ ($i\in\{1,\ldots, k-1\}$) arbitrary and $e_k=1$.
Suppose that
\begin{equation}
\label{pablo}
\Phi_{pq}^{e_1}\Phi_{p^2q}^{e_2}\cdots \Phi_{p^kq}^{e_k}=\mathrm P_T,
\end{equation}
with $T$ a numerical semigroup. Then $e_i=1$ for $1\le i\le k$ and 
$T=\langle p^k,q\rangle $.
\end{Lem}
\begin{proof} In case $e_i=1$ for $1\le i\le k$ the identity \eqref{pablo} holds with $S=\langle p^k,q\rangle$ by 
\eqref{Pa1a2} with $a_1=p^k$ and $a_2=q$. 
Since, modulo $x^2$, $\Phi_{p^mq}=1$ for $m\ge 2$ and 
$\Phi_{pq}=1-x$, we infer that $e_1=1$.
Suppose now we are not in the case
where $e_i=1$ for $1\le i\le k$, hence the largest integer $j_1$ with $e_{j_1}=1$ satisfies $1\le j_1<k$.
We let $j_2$ be the smallest integer such that $j_2>j_1$ and $e_{j_2}=1$. Since $e_k=1$, $j_2$ exists.
We now rewrite the left-hand side of \eqref{pablo} as
$$\mathrm P_{\langle p^{j_1},q\rangle}(x)\Phi_{pq}(x^{p^{j_2}})^{e_{j_2}}\cdots \Phi_{pq}(x^{p^{k}})^{e_k},$$
which by \eqref{basic-2} equals, modulo $x^{p^{j_2}+1}$,
$$\mathrm P_{\langle p^{j_1},q\rangle}(x)(1-x^{p^{j_2}}).$$
{}From this and \eqref{pablo} we infer
that $$\sum_{s\in \langle p^{j_1},q\rangle}x^s (1-x^{p^{j_2}})\equiv \mathrm H_T(x) \pmod{x^{p^{j_2}+1}}.$$
It follows that $p^{j_1}\in T$ and $p^{j_2}\not \in T$ and hence $T$ is not a numerical semigroup, contradicting
our assumption. 
\end{proof}

\begin{proof}[Proof of Theorem \ref{pnq}]
By \eqref{uitelkaar} with $m=p^nq$ and Lemma \ref{lem-basaal} we deduce that
 \begin{equation}
\label{pablo2}
\mathrm P_T=\Phi_{pq}^{e_1}\Phi_{p^2q}^{e_2}\cdots \Phi_{p^{n}q}^{e_n},
\end{equation}
with $0\le e_i \le 1$.
Since, modulo $x^2$, $\Phi_{p^iq}=1$ for $i\ge 2$ and 
$\Phi_{pq}=1-x$, we infer that $e_1=1$. Note that $e_n=1$, for
otherwise $d\mid p^{n-1}q$. The proof is concluded with the help of Lemma \ref{rijtje}. 
\end{proof}

\section{Acknowledgements}
\par This project was started in August 2013, when the first author 
carried out an internship at the Max Planck Institute for Mathematics,  under the guidance of the third author. He would like to thank the staff for the opportunity given and the very pleasant and productive stay. In January 2014 the first and second author were invited to the Max Planck Institute for 
Mathematics to work on this paper. They thank the staff for hospitality. The first and 
third author thank the second author for
inviting them to the INdAM meeting on numerical semigroups (Cortona, 2014), which gave them
a lot of additional motivation and inspiration. In September 2015 Caterina Viola did an internship at the Max Planck Institute for Mathematics, being supervised by the third author. The authors would like to thank her for the inspiring discussions which gave them further insight into the project. Last, not least, the authors thank the referee for the helpful remarks and suggested modifications.

\end{document}